\pdfoutput=1 
\documentclass[11pt]{amsart}
\usepackage{mathtools,amssymb}
\usepackage{graphicx}
\usepackage{comment,overpic,epsfig}
\usepackage[usenames,dvipsnames]{color}
\usepackage[hyphens]{url}
\usepackage[pagebackref,linktocpage=true,colorlinks=true,linkcolor=RoyalBlue,citecolor=BrickRed,urlcolor=RoyalBlue]{hyperref}
\usepackage[msc-links,abbrev]{amsrefs}

\newcommand{\R}{\mathbf{R}}
\newcommand{\ol}{\overline}

\newcommand{\C}{\mathcal{C}}
\renewcommand{\tilde}{\widetilde}
\renewcommand{\S}{\mathbf{S}}
\renewcommand{\phi}{\varphi}
\renewcommand{\epsilon}{\varepsilon}
\renewcommand{\leq}{\leqslant}
\renewcommand{\geq}{\geqslant}

\DeclareMathOperator{\inte}{int}
\DeclareMathOperator{\conv}{conv}
\DeclareMathOperator{\relint}{relint}
\DeclareMathOperator{\ave}{ave}

\DeclareMathOperator{\Emb}{Emb}
\DeclareMathOperator{\Imm}{Imm}
\DeclareMathOperator{\length}{length}

\theoremstyle{plain}
\newtheorem{theorem}{Theorem}[section]
\newtheorem{proposition}[theorem]{Proposition}
\newtheorem{lemma}[theorem]{Lemma}
\newtheorem{corollary}[theorem]{Corollary}
\theoremstyle{remark}

\theoremstyle{definition}
\newtheorem{note}[theorem]{Note}

\begin{document}

\vspace*{-0.5in}
	
\title[Curves and knots of constant torsion]{$h$-Principles for curves and knots\\ of constant torsion}

\author{Mohammad Ghomi}
\address{School of Mathematics, Georgia Institute of Technology, Atlanta, Georgia 30332}
\email{ghomi@math.gatech.edu}
\urladdr{https://ghomi.math.gatech.edu}

\author{Matteo Raffaelli}
\address{School of Mathematics, Georgia Institute of Technology, Atlanta, Georgia 30332}
\email{raffaelli@math.gatech.edu}
\urladdr{https://matteoraffaelli.com}

\date{Last revised on \today}
\thanks{The first-named author was supported by NSF grant DMS-2202337.}

\begin{abstract}
We prove that curves of constant torsion satisfy the $\C^1$-dense $h$-principle in the space of immersed curves  in Euclidean space. In particular, there exists a knot of constant torsion in each isotopy class.  Our methods, which involve convex integration and degree theory, quickly establish these results for curves of constant curvature as well.
\end{abstract}

\maketitle

\section{Introduction}
Curves of constant torsion, which occur naturally as elastic rods, have long been studied \cite{koenigs1887,wiener1974, wiener1977, ivey2000,calini-ivey1998,bates-melko2013,ivey-singer1999,bray-jauregui2015,musso2001}, and some knotted examples have been found by various means. Here we construct  knots of constant torsion in every isotopy class by adapting the convex integration   \cite{gromov:PDR,spring1998,geiges2003} techniques developed  for curves of constant curvature \cite{ghomi2007-h,ghomi-raffaelli2024}.  
To state our main result, let $\Gamma$ be an interval $[a,b]\subset\R$  or topological circle $\R/((b-a)\textbf{Z})$, and $\C^\alpha(\Gamma,\R^3)$ be the space of $\C^\alpha$ curves $f\colon \Gamma\to\R^3$ with its standard norm 
$|\cdot|_\alpha$.  Let $\Imm^{\alpha\geq 1}(\Gamma,\R^3)\subset \C^\alpha(\Gamma,\R^3)$ consist of curves with speed $|f'|\neq 0$. If $|f'|=1$, the \emph{curvature} and \emph{torsion} of $f$ are given by $\kappa\coloneqq   |f''|$ and $\tau\coloneqq  \det(f', f'', f''')/\kappa^2$ respectively.

\begin{theorem}\label{thm:main}
Let $f\in\textup{Imm}^{\alpha\geq 4}(\Gamma,\R^3)$ be a unit speed curve with $\kappa$, $\tau>0$, and $p_i\in\Gamma$ be a finite collection of points. Then for  any $\epsilon>0$ there exists a unit speed curve $\tilde f\in\textup{Imm}^{\alpha-1}(\Gamma,\R^3)$ with $\tilde\kappa>0$ and $\tilde\tau=constant$ such that $|\tilde f- f|_1\leq\epsilon$ and $\tilde f$ is tangent to $f$ at $p_i$.
\end{theorem}

If $\epsilon$ is sufficiently small, then $h_t\coloneqq  (1-t)f+t \tilde f$, $t\in[0,1]$,  is a homotopy  in $\Imm^{\alpha-1}(\Gamma,\R^3)$. In the terminology of Gromov or Eliashberg \cite{gromov:PDR, cem2024}, this constitutes  a  $\C^1$-dense $h$-principle for curves of constant torsion. 
 Let $\Emb^\alpha(\Gamma,\R^3)\subset\Imm^\alpha(\Gamma,\R^3)$ be the space of injective  curves, which are called \emph{knots} when $\Gamma$ is a circle.  If $f\in\Emb^1(\Gamma,\R^3)$ and $\epsilon$ is sufficiently small, then $h_t$ is an isotopy. Thus, since  curves in $\Emb^{\infty}(\Gamma,\R^3)$ with $\kappa$, $\tau>0$ are dense in $\Emb^{1}(\Gamma,\R^3)$,  we obtain:

\begin{corollary}
Every knot  $f\in\Emb^1(\Gamma,\R^3)$ is isotopic in $\Emb^1(\Gamma,\R^3)$ to a knot $\tilde f\in\Emb^\infty(\Gamma,\R^3)$ with $\tilde\kappa>0$ and $\tilde\tau=constant$. 
\end{corollary}

Analogous results for curvature were established in \cite{ghomi2007-h}, see also \cite{wasem2016,ghomi-raffaelli2024} for related $h$-principles, and \cite{mcatee2007,koch:constant} for earlier constructions. As in  \cite{ghomi2007-h}, we prove Theorem \ref{thm:main} by reducing it to a problem for spherical curves (Propositions \ref{prop:loc} and \ref{prop:T}). More explicitly, we deform the \emph{tantrix} $T:=f'$ of $f$ to a longer spherical curve $\tilde T$ with $|\tilde T-T|_0\leq\epsilon$, and then integrate $\tilde T$ to obtain $\tilde f$. For $\tilde \tau$ to be constant, the product $\tilde k\tilde v$ must be constant (Lemma \ref{lem:reparam}), where $\tilde k$ is the geodesic curvature and $\tilde v$ is the speed of $\tilde T$. Furthermore, for $\tilde f$ to be tangent to $f$ at $p_i$ we need to have $\int \tilde T=\int T$ on every interval between $p_i$. We will show that these requirements can be met via basic convex geometry  together with degree theory (Lemma \ref{lem:F}), which makes the arguments significantly shorter than those in  \cite{ghomi2007-h}, although less explicit. 

Constructing submanifolds with prescribed tangential directions has been a major theme in $h$-principle theory, e.g., see \cite{ghomi2011} and references therein. In particular see \cite{ghomi:shadow,ghomi:shadowII} for more results and applications of curves with prescribed tantrices.

\begin{note}
Our methods also establish the analogue of Theorem \ref{thm:main} for curvature, with obvious simplifications since $\tilde T$ would only need to have constant speed. In particular Lemma \ref{lem:reparam} below is not needed. Furthermore, in Proposition \ref{prop:loc} we may replace the condition $\tilde \tau=c$ with $\tilde\kappa=c$, and in Proposition \ref{prop:T} replace $\tilde k\tilde v=c$ with $\tilde v=c$. The  proofs will then proceed along the same lines, with only some abbreviations.
\end{note}

\section{Reparametrization of the Tantrix}\label{sec:reparam}
We begin by constructing constant torsion curves with a prescribed tantrix image. 
Set $I\coloneqq   [a,b]$, and $|I|:=b-a$. Let $f\in\Imm^3(I,\R^3)$ be a curve with $|f'|=1$, and set $v\coloneqq   |T'|=\kappa$. If $v\neq 0$, then $T\in\Imm^2(I,\S^2)$, and 
$N\coloneqq   T'/v$, $B\coloneqq   T\times N$ generate the Frenet frame $(T,N,B)$. Then we may compute that
$$
\tau=\langle N',B\rangle=\frac{\langle v  T'' - v' T',B\rangle}{v^{2}}=\frac{\langle T''-v'N, B\rangle}{v}= \frac{\langle T'', B\rangle}{v}=k v,
$$
where $ k\coloneqq   \langle T'', B\rangle/v^2$ is the geodesic curvature of $T$.  We say  $\tilde T$ is a \emph{reparametrization} of $T$ if $\tilde T=T\circ\phi$ for an increasing diffeomorphism $\phi\colon I\to I$.  Standard ODE theory yields:

\begin{lemma}\label{lem:reparam}
Let $T\in \textup{Imm}^{\alpha\geq 3}(I,\mathbf{S}^{2})$ be a curve with $k> 0$. Then $T$ admits a  unique reparametrization $\tilde T=T\circ\phi\in \textup{Imm}^{\alpha-1}(I,\mathbf{S}^{2})$ such that
$
\tilde k\, \tilde v
$
is constant. Furthermore, the mapping $T\to\tilde T$ is continuous.
\end{lemma}
\begin{proof}
For $c>0$, the equation
$
\tilde k\, \tilde v=c
$
may be rewritten as 
$
(v\circ\phi)\varphi' = c/( k \circ \varphi)
$
by the chain rule and  invariance of geodesic curvature. 
Since $\alpha\geq 3$, $v$ and $ k$ are Lipschitz, and may be extended to Lipschitz functions on $\R$ without loss of regularity. So we arrive at the  initial value problem
\begin{equation*}
\begin{cases}
\varphi' = F_c(\phi),\\
\varphi(a)=a;
\end{cases}
\quad\quad\text{where}\quad\quad
F_c(\cdot)\coloneqq   \frac{c}{v(\cdot)k(\cdot)}.
\end{equation*}
Since  $F_c\colon\R\to\R$ is Lipschitz, for every $c$ there exists a unique solution  $\phi_c\colon I\to\R$ by the Picard--Lindel\"{o}f theorem. Since $F_c$ is $\C^{\alpha-2}$, $\phi_c$ is $\C^{\alpha-1}$, and since $\varphi_c'\neq 0$,  $\varphi_c$ is a diffeomorphism onto its image. Note that $c\mapsto \phi_c(b)$ is a continuous monotonic function since $\phi_c$ depends continuously on $c$ and $F_c$ varies monotonically with $c$. Furthermore, $\phi_c(b)$ can be made arbitrarily small or large along with $F_c$.
Hence $\phi_{c_0}(b)=b$ for a unique $c_0$, which yields the desired reparametrization. Finally, $\phi$ depends continuously on $F_c$, which in turn depends continuously on $T$. Thus $\phi$ depends continuously on $T$, and so $\tilde T$ depends continuously on $T$.
\end{proof}

Suppose now that $f\in\Imm^{\alpha\geq 4}(I,\R^3)$ is a curve with $\kappa$, $\tau>0$ and tantrix $T$. Then $T\in\Imm^{\alpha-1}(I,\S^2)$ with geodesic curvature $k=\tau/\kappa>0$. So we may apply the above lemma to obtain the reparametrization 
$\tilde T\in\Imm^{\alpha-2}(I,\S^2)$. Then 
\begin{equation}\label{eq:tildef}
\tilde f(t)\coloneqq   f(a)+\int_a^t\tilde T(u)du
\end{equation}
is a $\C^{\alpha-1}$ curve of constant torsion $c$ with $\tilde T(I)=T(I)$. Since $c=\tilde k\tilde v=(k\circ\phi)\tilde v$, and $L:=\length(T)=\length(\tilde T)=\int_I\tilde v$, we obtain the following estimate:
\begin{equation}\label{eq:c}
\frac{L}{|I|}\,\textup{min}_{I}(k)\;\;\leq\;\;c=\frac{L}{|I|\,\ave_{I}\big(1/(k\circ\phi)\big)}\;\;\leq\;\; \frac{L}{|I|}\,\textup{max}_{I}(k),
\end{equation}
where $\ave_I (\cdot)\coloneqq\int_I(\cdot)/|I|$.

\section{Reduction to Spherical Curves}\label{sec:spherical}

Here we use tantrices to reduce Theorem \ref{thm:main} to a problem for spherical curves. First we show that Theorem \ref{thm:main} follows from a more geometric local result. We say that a constant $c$ is \emph{arbitrarily large} if it can be chosen from an interval $[c_0,\infty)$.

\begin{proposition}\label{prop:loc}
Let $ f\in\textup{Imm}^{\alpha\geq 4}(I,\R^3)$ be a curve with $\kappa$, $\tau> 0$ and $V$ be an open neighborhood of $T(I)$ in $\S^2$. Then  there exists a unit speed curve $\tilde f\in\textup{Imm}^{\alpha-1}(I,\R^3)$ with $\tilde\kappa>0$ and $\tilde\tau=c$, for $c$ arbitrarily large, such that $\tilde T(I)\subset V$, $ f=\tilde f$ on $\partial I$,  and $T(U)=\tilde T(\tilde U)$ for some open neighborhoods $U$, $\tilde U$ of $\partial I$ in $I$. 
\end{proposition}

Proposition~\ref{prop:loc} implies Theorem~\ref{thm:main} as follows. Let $I_i$ be a partition of $\Gamma$ into intervals such that $\partial I_i$ include the prescribed points $p_j$. Choose $I_i$ so small that $T(I_i)$ lies in the interior $V_i$ of a disk  of radius $\epsilon/2$ in $\S^2$. Applying Proposition~\ref{prop:loc} to $f_{i}:=f|_{I_{i}}$, we obtain $\C^{\alpha-1}$ curves $\tilde{f}_{i}$ with $\tilde \tau_i=c_i$, $\tilde{f}_{i}=f_{i}$ on $\partial I_{i}$, $\tilde T_i(I_i)\subset V_i$, and $T_{i}(U_{i}) = \tilde{T}_{i}(\tilde{U}_{i})$ for open neighborhoods $U_{i}, \tilde{U}_{i}$ of $\partial I_{i}$ in $I_{i}$. Note that since $c_i$ are arbitrarily large, they may assume the same value $c$. Define $\tilde f\colon\Gamma\to\R^3$ by setting $\tilde{f} \coloneqq  \tilde{f}_{i}$ on $I_{i}$. Then $\tilde{f}$ is $\C^{0}$, since $\tilde{f} = f$ on $\partial I_{i}$. Furthermore, since $\tilde{f}_{i}'=\tilde{T}_{i}$ and $\tilde{T}_{i} = T_{i}$ on $\partial I_{i}$, $\tilde{f}$ is $\C^{1}$, and so $\tilde T=\tilde f'$ is well-defined. Note that the condition $T_{i}(U_{i}) = \tilde{T}_{i}(\tilde{U}_{i})$ yields open neighborhoods $W_i$ and $\tilde W_i$ of $\partial I_i$ in $\Gamma$ such that $\tilde{T}(\tilde W_i) = T(W_i)$. Since $\tilde{T}|_{I_i}=\tilde T_i=\tilde f_i'$ is $\C^{\alpha-2}$, it follows that $\tilde T$ is a reparametrization of $T$ with speed $c/k$ near $\partial I_i$. So, by Lemma \ref{lem:reparam}, $\tilde{T}$ is $\C^{\alpha-2}$, or $\tilde f$ is $\C^{\alpha-1}$. Thus, $\tilde\tau$ is well-defined and is equal to $c$. Finally, since $\tilde T(I_i)\subset V_i$, we have $|\tilde T-T|_0\leq\epsilon$, which yields $|\tilde f-f|_1\leq\epsilon$, as desired.
 
It remains then to prove Proposition \ref{prop:loc}, which we reduce in turn to a more basic result.
Let $\conv(f)$ denote the convex hull of $f$. 
Note that $\ave_I(f)$  lies in the relative interior $\relint \conv(f)$ of $\conv(f)$ \cite[Lem.~2.1]{ghomi2007-h}, i.e., the interior of $\conv(f)$ within the affine hull of $f$. We say that $f$ is \emph{nonflat} provided that $\conv(f)$ has interior points, or $\inte(\conv(f))\neq\emptyset$. Note that $\relint \conv(f) = \inte \conv(f)$ when $f$ is nonflat.

\begin{proposition}\label{prop:T}
Let $T\in\textup{Imm}^{\alpha\geq 3}(I,\S^2)$ be a nonflat curve with $k> 0$, $V\subset\S^2$ be an open neighborhood of $T(I)$, and $x_0\in\inte(\conv(T))$. Then there exists a curve $\tilde T\in\textup{Imm}^{\alpha-1}(I,\S^2)$ with $\tilde  k\tilde v=c$, for $c$ arbitrarily large, $\tilde T(I)\subset V$, $T(U)=\tilde T(\tilde U)$ for some open neighborhoods $U$, $\tilde U$ of $\partial I$ in $I$, and $\ave _I(\tilde T)=x_0$. 
\end{proposition}

Proposition \ref{prop:T} implies Proposition \ref{prop:loc} as follows.  After a perturbation of $f$ on a compact set in the interior of $I$  we may assume that $T$ is nonflat. Let $x_0\coloneqq \ave _I(T)$, $\tilde T$ be the corresponding curve given by Proposition \ref{prop:T}, and $\tilde f$ be given by \eqref{eq:tildef}. Then  $\tilde f'=\tilde T$. So $|\tilde f'|=1$, $\tilde\kappa=|\tilde T'|>0$, and $\tilde\tau=\tilde k\tilde v=c$ as desired. Finally,  $\int_I T=|I|\ave_I (T)=|I|\ave_I(\tilde T)=\int_I \tilde T$ which ensures that $f=\tilde f$ on $\partial I$ and completes the argument.

\section{Controlling the Average}\label{sec:proof}
Here we establish Proposition \ref{prop:T}, which completes the proof of Theorem \ref{thm:main} as discussed above. First we record the
following basic fact. 

\begin{lemma}\label{lem:F}
Let $B\subset\R^n$ be a ball of radius $R$ centered at $x_0$, and $F\colon B \to \mathbf{R}^{n}$ be a continuous map. If $| F(x)-x| <R$ for all $x\in \partial B$, then $x_0 \in F(B)$.
\end{lemma}

\begin{proof}
For $t\in[0,1]$, let $F_{t}(x) \coloneqq   (1-t)x + t F(x)$, and set $f_{t}\coloneqq   F_{t}\rvert_{\partial B}$. Since $| F(x)-x| <R$, for $x\in\partial B$, $x_0\not\in f_{t}(\partial B)$. So $\tilde{f} _{t}\coloneqq    f_{t}/| f_{t}|\colon \partial B \to \partial B$  is well-defined. Since $\tilde{f}_{0}$ is the identity map, $\textup{deg}(\tilde f_0)=1$. Thus, since $\tilde{f}_{t}$ is a homotopy, $\textup{deg}(\tilde f_1)=1$. But $\tilde f_1=F/|F|$. So if $x_0 \notin F(B)$, $\tilde f_1$ may be extended to $B$, which implies that it is homotopic to a constant map; therefore $\textup{deg}(\tilde f_1)=0$, which is a contradiction.
\end{proof}

Now we prove Proposition \ref{prop:T}. By Steinitz's theorem \cite[Thm.~1.3.10]{schneider2014}, there is a minimal set of points $v_i \in T$, $i=1,\dots, N\leq 6$, such that $x_0\in\inte(\conv(\{v_i \}))$. Let $B\subset\inte(\conv(\{v_i \}))$ be a ball of radius $R$ centered at $x_0$. Then, by Carath\'{e}odory's theorem \cite[Thm.~1.1.4]{schneider2014}, for each $x\in B$ there are constants $\lambda_i(x)>0$, with $\sum_i\lambda_i(x)=1$, such that 
$
x=\sum_i \lambda_i(x)v_i.
$
 By a theorem of Kalman \cite{kalman1961}, we may assume that $\lambda_i\colon B\to\R$ are continuous. Then $\ol\lambda:=\min_B\{\lambda_i\}>0$.

Let $C_i\subset V$ be circles of radius $r<R/(4N)$ which are tangent to $T$ at $v_i $, and lie on the convex side of $T$.  Let $\tilde k_C$ be the geodesic curvature of $C_i$, which depends only on $r$. Choose $r$ so small that $\tilde k_C\geq \max_I(k)$. Then, after a perturbation of $T$, we may assume that a neighborhood of $v_i $ in $T$ lies on $C_i$, see Figure \ref{fig:loop}. 
\begin{figure}[h]
\begin{overpic}[height=1in]{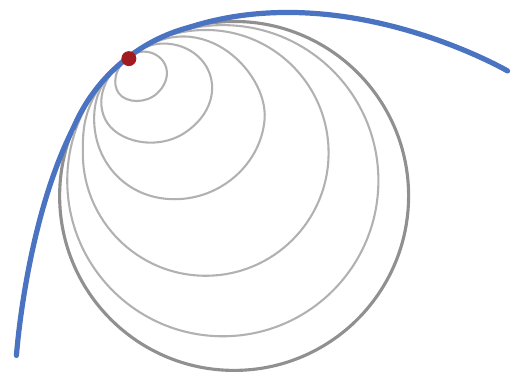}
\put(-6.5,5){\small $T$}
\put(15,65){\small $v_i$}
\put(76,15){\small $C_i$}
\put(30,25){\small $C_{i,\ell}$}
\end{overpic}
\caption{}\label{fig:loop}
\end{figure}
Let $C_{i,\ell}$ be a continuous family of nested $\C^\alpha$ curves of length $0<\ell<\length(C_i)$ and nondecreasing curvature which contain an open neighborhood of $v_i $ in $T$ and shrink to $v_i $ as $\ell\to 0$. Then for any $\ell_i>0$ we can construct a unique composite loop of length $\ell_i$ at $v_i $ as follows. Let $m$ be the largest integer with $m\length(C_i)\leq \ell_i$, and set $\ell:= \ell_i - m\length(C_i)$. Then make $m$ full laps around $C_i$ followed by one lap around $C_{i,\ell}$.

Set $L\coloneqq   \length(T)$, and choose $\tilde L>L$. For $x\in B$, let $\ol T_x\in\Imm^\alpha(I,\S^2)$ be the constant speed curve of length $\tilde L$ which traces the image of $T$ plus loops of length 
$$
\ell_i(x)\coloneqq   \lambda_i(x)(\tilde L-L)
$$ 
at $v_i $. Then $x\mapsto \ol T_x$ is continuous with respect to the $\C^0$-norm on $\C^0(I,\S^2)$. Next let $\tilde T_x$ be the reparametrizations of $\ol T_x$ given by Lemma \ref{lem:reparam}, and note that $\ol T_x\mapsto\tilde T_x$ is continuous with respect to the $\C^0$-norm. Thus the mapping
$$
B\;\ni x\;\xmapsto{\;\;\;\;\;F\;\;\;\;}\,\ave_I(\tilde T_x)\,\in\,\R^3
$$ 
is continuous. We may assume that $|I|=1$. Then by \eqref{eq:c} and since $\tilde k_C\geq \max_I(k)$,
$$
c=\tilde L/\ave_{I}\big(\,1/\tilde k_x\big)\geq \tilde L\,\textup{min}_I(\tilde k_x)=\tilde L\,\textup{min}_I(k).
$$
So $c\to\infty$  as $\tilde L\to\infty$. We will show that if $\tilde L$ is sufficiently large, then $|F(x)-x|<R$, which will complete the proof by Lemma \ref{lem:F}.

Let $\tilde T^i_x$ be the part of $\tilde T_x$ which forms the loop at $v_i $, and $I^i_x\subset I$ be the subinterval such that $\tilde T^i_x=\tilde T|_{I^i_x}$. Set $I'_x\coloneqq   I-\cup_i  I^i_x$. Then
$$
F(x)
=
\sum_i |I^i_x| \ave_{I^i_x}(\tilde T^i_x)
+
|I'_x|\ave_{I'_x}(\tilde T_x).
$$
Since $\ave_{I^i_x}(\tilde T^i_x)\in\textup{conv}(\tilde T^i_x)\subset\conv(D_i)$, where $D_i\subset\S^2$ is the small disk bounded by $C_i$, we have $\lvert\ave_{I^i_x}(\tilde T^i_x)-v_i \rvert\leq 2r<R/(2N)$. So subtracting $\sum_i |I^i_x| v_i $ from both sides of the above equation yields
\begin{equation}\label{eq:F}
\Big|\,F(x)-\sum_i |I^i_x| v_i \,\Big|
<
\frac{R}{2}
+|I'_x| \big\lvert\ave_{I'_x}(\tilde T_x)\big\rvert.
\end{equation}
Now let $\tilde L\to\infty$.
Note that 
$
|I'_x|=L/\ave_{I'_x}(c/\tilde k_x),
$
and $\min_{I'_x}(\tilde  k_x)=\min_I(k)>0$, since $\tilde T_x(I'_x)=T(I)$.
So $|I'_x|\to 0$.  But, since $\ave_{I'_x}(\tilde T_x)\in\conv(T)$, $\lvert\ave_{I'_x}(\tilde T_x)\rvert$ is  bounded above.
So the right hand side of \eqref{eq:F} converges uniformly to $R/2$.  Next note that
$$
|I^i_x|=\frac{\ell_i(x)}{c\,\ave_{I^i_x}\big(\,1/\tilde k_x\big)}=\frac{\lambda_i(x)(\tilde{L}-L)\ave_{I}\big(\,1/\tilde k_x\big)}{\tilde L\,\ave_{I^i_x}\big(\, 1/\tilde k_x\big)}.
$$
We have $\ell_i\to\infty$  since $\lambda_i\geq\ol\lambda>0$.
So $\ave_{I^i_x}(1/\tilde k_x)\to 1/\tilde k_C$. Since $|I'_x|\to 0$, $\ave_{I}(1/\tilde k_x)\to 1/\tilde k_C$ as well.
Thus $|I^i_x|\to \lambda_i(x)$.
 So the left hand side of \eqref{eq:F} converges uniformly to
$
|F(x)-x|,
$
 which completes the proof.

\bibliographystyle{amsplain}
\bibliography{references}
\end{document}